\documentclass[12pt,a4paper]{amsart}         
\usepackage[margin=3cm]{geometry}

\usepackage{amsmath,amsfonts,amssymb,amsthm}
\usepackage[abbrev]{amsrefs}
\usepackage{mathrsfs,epsfig,epstopdf,url}
\usepackage{xcolor}
\usepackage{float}
\usepackage{tikz}
\restylefloat{table}
\usepackage{multirow}
\usepackage{array}

\usepackage{pgfplots}

\theoremstyle{definition}
\newtheorem{definition}{Definition}[section]

\theoremstyle{plain}

\newtheorem{lemma}[definition]{Lemma}
\newtheorem{proposition}[definition]{Proposition}
\newtheorem{problem}[definition]{Problem}

\theoremstyle{remark}
\newtheorem*{remark}{Remark}

\newcommand{\E}{\mathsf{E}}
\newcommand{\D}{\mathsf{D}}
\newcommand{\Var}{\mathsf{Var}}

\newcommand{\bm}[1]{\boldsymbol{#1}}
\DeclareMathOperator{\supp}{supp}

\DeclareMathOperator{\Frob}{Frob}
\DeclareMathOperator{\Aut}{Aut}

\title[Subfield Subcodes of Hermitian Codes]{Estimating The Dimension Of The Subfield Subcodes of Hermitian Codes}

\author{Sabira El Khalfaoui}
\address{Bolyai Institute \\
	University of Szeged \\
	Aradi v\'ertan\'uk tere 1\\
	H-6720 Szeged, Hungary}
\email{sabira@math.u-szeged.hu}

\author{G\'abor P. Nagy}
\address{Department of Algebra \\
	Budapest University of Technology and Economics\\
	Egry J\'ozsef utca 1\\
	H-1111 Budapest, Hungary}
\address{Bolyai Institute \\
	University of Szeged \\
	Aradi v\'ertan\'uk tere 1\\
	H-6720 Szeged, Hungary}
\email{nagyg@math.u-szeged.hu}

\keywords{AG code, Hermitian code, subfield subcode, extreme value distribution}
\subjclass[2010]{ 11T71, 14G50, 94B27}

\begin{document}
\begin{abstract}
In this paper, we study the behavior of the true dimension of the subfield subcodes of Hermitian codes. Our motivation is to use these classes of linear codes to improve the parameters of the McEliece cryptosystem, such that key size and security level. The McEliece scheme is one of the promising alternative cryptographic schemes to the current public key schemes since in the last four decades, they resisted all known quantum computing attacks. By analyzing computational data series of true dimension, we concluded that they can be estimated by the extreme value distribution function. 
\end{abstract}
\maketitle

\section{Introduction}

Recently, there has been a big amount of research addressed to quantum computers that use quantum mechanical techniques to solve hard problems in mathematics \cite{google2019quantum}. The existence of these powerful machines threaten many of the public-key cryptosystem that are widely in use. Combined with Shor's algorithms \cite{shor1999}, this would risk the confidentiality and integrity of today's digital communications. Post-quantum cryptography aims to construct and develop cryptosystem that must resist against quantum computing attacks.

McEliece \cite{mceliece1978public} introduced the first code-based public-key cryptosystem in 1978, where he employed error correcting codes to generate the public and private key with security relying on two aspects: NP-completeness of decoding linear codes and the distinguishing of the chosen codes. The original McEliece scheme was constructed with binary Goppa codes which are the subfield subcodes of generalized Reed-Solomon codes. Even today, this proposal represents a good candidate for post-quantum cryptography \cite{nnn}. There has been several attempts to find appropriate classes of codes and their parameters, which give rise to a secure and effective cryptosystem, for more details see \cites{Niebuhr2012,loidreau2000strengthening}. In this paper, we study the possibility of the application of subfield subcodes of Hermitian codes in the McEliece scheme. More precisely, we do the first step by investigating the true dimension of these codes for a broad spectrum of parameters, for partial results  see \cites{khalfaoui2019dimension,pinero2014subfield}. Our main observation is that the true dimension of subfield subcodes of Hermitian codes can be estimated by the extreme value distribution function.

In the literature, several attacks have been proposed against McEliece cryptosystem in general, and against McEliece systems based on AG codes, see \cites{baldi2013security,loidreau2000strengthening,Couvreur2017}. Attacks can be divided into two classes: structural, or key recovery attacks, aimed at recovering the secret code, and decoding, or message recovery attacks, aimed at decrypting the transmitted ciphertext. The generic decoding attack against the McEliece scheme is the information set decoding (ISD) algorithm. The most recent and most effective structural attack against AG code based McEliece systems is the Schur product distinguisher.

The structure of this paper is as follows. In section \ref{sec:background}, we review the necessary backgrounds to define subfield subcodes, algebraic geometry codes and Hermitian codes. In section \ref{sec:moments}, we introduce some tools borrowed from statistics in order to handle our computed data on the true dimension of subfield subcodes of Hermitian codes, the latter being presented in section \ref{sec:subcodes}. Our main result is Proposition \ref{pr:main} in section \ref{sec:fitting} which shows the excellent fitting properties of the extreme value distribution to our measurements. In section \ref{sec:application}, we applied this estimate to study the development of the key size of Hermitian subfield subcodes.

\section{Backgrounds, formulas}
\label{sec:background}

In this section, we give an overview on subfield subcodes, AG codes and some of their properties, to find full details please refer to the monographs \cites{hirschfeld2008algebraic,stepanov2012codes,stichtenoth2009algebraic}. Our terminology on coding theory is standard, see \cites{stepanov2012codes,hoholdt1995decoding}. In particular, with an $\mathbb{F}_q$-linear code of length $n$, we mean a linear subspace of $\mathbb{F}_q^n$. 

\subsection{Subfield subcodes}

Let $h$ be an integer and $r, q$ be prime powers with $q=r^h$. Then $\mathbb{F}_r$ is a subfield of $\mathbb{F}_q$ and the field extension $\mathbb{F}_q/\mathbb{F}_r$ has degree $h$. Let $C$ be an $\mathbb{F}_q$-linear code of length $n$ and dimension $k$. The $\mathbb{F}_q/\mathbb{F}_r$ subfield subcode of $C$ is defined by
\[
C|_{\mathbb{F}_r}=C\cap\mathbb{F}_r^n.
\]
The trace polynomial $\mathrm{Tr}(x)=x+x^r +\dots+x^{r^{h-1}}$ defines a map $\mathbb{F}_q \to \mathbb{F}_r$, which can be extended to a map $\mathbb{F}_q^n \to \mathbb{F}_r^n$ component wise. The trace code of the linear code $C$ is 
\[
\mathrm{Tr}(C)= \left\lbrace \mathrm{Tr}(c) \mid  c \in C \right\rbrace.
\]
Clearly, both the subfield subcode and the trace code are $\mathbb{F}_r$-linear codes of length $n$. However, it is in general very hard to determine the true dimension of these new codes. The fascinating result given by Delsarte \cite{Del} in 1975 plays a key role for studying the class of the subfield subcodes of linear codes. It established a closed link between subfield subcodes and trace codes:

\begin{equation*} 
(C|_{\mathbb{F}_r})^\perp = \mathrm{Tr}(C^{\perp}).
\end{equation*}
V\'eron \cite{V05} used this equation to give the exact dimension formula
\begin{equation}
\dim_{\mathbb{F}_r}(C|_{\mathbb{F}_r}) = n-h(n-k)+\dim_{\mathbb{F}_r} \ker (\mathrm{Tr}).
\end{equation}
In particular, we have the trace bound 
\begin{equation}
\dim_{\mathbb{F}_r}(C|_{\mathbb{F}_r}) \geq n-h(n-k).
\end{equation}

\subsection{Algebraic geometry codes}

In this section, we give an overview on the construction of algebraic geometry (AG) codes, which is a version of V.D. Goppa's original construction, since there are many ways to produce linear codes from algebraic curves. Also we give some details on the properties, parameters and duality of AG codes. AG codes are linear codes that use algebraic curves and finite fields for their construction. The construction can be done by evaluating functions (elements of the function field) or by computing residues of differentials. Our notation and terminology on algebraic plane curves over finite fields, their function fields, divisors and Riemann-Roch spaces are standard, see for instance \cites{hirschfeld2008algebraic,menezes2013applications,stichtenoth2009algebraic}.

Let $q$ be a prime power and $\mathbb{F}_q$ be the finite field of order $q$. Let $\mathcal{X}$ be an algebraic curve i.e. an affine or projective variety of dimension one, which is absolutely irreducible and nonsingular and whose
defining equations are (homogeneous) polynomials with coefficients in $\mathbb{F}_q$. Let $g$ be the genus of $\mathcal{X}$ and denote by ${\mathbb{F}}_q(\mathcal{X})$ the function field of $\mathcal{X}$. For a divisor of $D$ of ${\mathbb{F}}_q(\mathcal{X})$, the Riemann-Roch space is
\[\mathscr{L}(D)=\{f\in {\mathbb{F}}_q(\mathcal{X}) \mid  (f)\succcurlyeq -D\} \cup \{0\},\]
where $(f)$ is the principal divisor of $f$. The dimension $\ell(D)$ of $\mathscr{L}(D)$ is given by the Riemann-Roch Theorem \cite{stichtenoth2009algebraic}*{Theorem 1.1.15}:
\begin{equation} \label{eq:rrth}
\ell(D)= \ell(W-D) + \deg D - g + 1,
\end{equation}
where $W$ is a canonical divisor of ${\mathbb{F}}_q(\mathcal{X})$. Let $G$ and $D$ be two divisors of ${\mathbb{F}}_q(\mathcal{X})$ such that $D = P_1 + \cdots + P_n$ is the sum of $n$ distinct rational places of ${\mathbb{F}}_q(\mathcal{X})$ and $P_i\not\in \supp(G)$ for any $i$. With these data, two types of algebraic geometry codes can be constructed:
\begin{align*}
C_L(D,G)&=\left\lbrace \left( f(P_1),\cdots,f(P_n) \right) \mid f\in \mathscr{L}(D)   \right\rbrace ,\\
C_\Omega(D,G)&= \left\lbrace \left( res_{P_1}(\omega),\cdots,res_{P_n}(\omega)\right) \mid \omega \in \Omega(G-D)\right\rbrace .
\end{align*}
The codes $C_L(D,G)$ and $C_\Omega(D,G)$ are called the \textit{functional} and the \textit{differential codes,} respectively. These two codes are dual to each other. Moreover, the differential code $C_\Omega(D,G)$ is equivalent with the functional code $C_L(D,W+D-G)$. In particular, they have the same dimension and minimum distance, even though this equivalence does not preserve all important properties of the code. The formula 
\[k=\ell(G)-\ell(G-D)\]
for the dimension $k$ of $C_L(D,G)$ follows from the Riemann-Roch Theorem, which also provides a lower bound $\delta_\Gamma=n-\deg(G)$ for its minimum distance. The integer $\delta_\Gamma$ is called the \textit{Goppa designed minimum distance} of the AG code. 

We illustrate the behavior of the dimension $k$ of $C_L(D,G)$ depending on the degree of the divisor $G$ by Figure \ref{fig:dims}. In fact, \eqref{eq:rrth} implies the exact value $k=\deg(G)-g+1$ provided $2g-2 < \deg(G) < n$. Furthermore, if $\deg(G)>n+2g-2$, then $k=n$. In the intervals $[0,2g-2]$, and $[n,n+2g-2]$, the dimension depends on the specific structure of the divisor $G$. 

\begin{figure}
\caption{Dimension and designed minimum distance of AG codes}
\label{fig:dims}
\begin{center}
	\begin{tikzpicture}[scale=0.5]
	\draw[thick,->] (-0.5,0) -- (13,0) node[above right] {$\deg(G)$};
	\draw[thick,->] (0,-0.5) -- (0,10) node [right] {$\dim$};
	\draw[dashed] (0,9) -- (12,9) -- (12,0);
	\draw[dashed] (3.0,0) -- (3.0,9);
	\draw[dashed] (9,0) -- (9,9);
	\draw[thick,dotted] (1,-0.5) -- (11,9.5) node[above] {$\dim=\deg(G)-g+1$};
	\draw[blue,ultra thick] (3.0,1.5) -- (9,7.5);
	\draw[red,ultra thick,dashed] plot[domain=0:3] (\x, {\x^2/6});
	\draw[red,ultra thick,dashed] plot[domain=9:12] (\x, {9-(\x-12)^2/6});
	\draw[dotted] (0,9) -- (-1,10) node[above] {$\delta_\Gamma=n-\deg(G)$};
	\draw[green,ultra thick] (0,9) -- (9,0);
	\fill (0,9) circle (1mm) node[left=1mm] {$n$};
	\fill (0,0) circle (1mm) node[below=3mm,left=1mm] {$0$};
	\fill (3,0) circle (1mm) node[below=1mm] {$2g\!-\!2$};
	\fill (9,0) circle (1mm) node[below=1.8mm] {$n$};
	\fill (12,0) circle (1mm) node[below=1mm] {$n\!+\!2g\!-\!2$};
	\end{tikzpicture}
\end{center}
\end{figure}
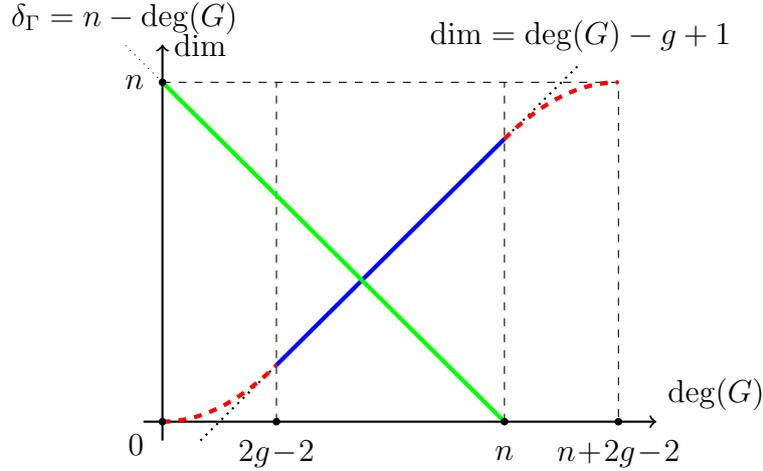

\subsection{On the decoding of AG codes}
Algebraic geometry codes are a generalization of Reed-Solomon codes, then it is not extraordinary that they benefit from similar decoding algorithms. The work on the decoding of AG codes seems to begin in 1986 when Driencourt gave a first decoding algorithm for codes on elliptic curves of characteristic $2$ \cite{driencourt1985some} correcting $\lfloor (\delta_\Gamma - 1)/2\rfloor$ errors. By generalizing the work of Arimoto and Peterson \cite{peterson1960encoding} on employing a locator polynomial to decode Reed-Solomon codes, Justesen, Larsen, Jensen, Havemose and H\o hold published \cite{justesen1989construction} in 1989 a decoding algorithm for a larger class of AG codes, which can correct up to $\lfloor (\delta_\Gamma -g -1)/2  \rfloor$ errors, moreover in improved version \cite{justesen1992fast} the error capability is increased to $\lfloor (\delta_\Gamma - g/2 -1 )/2 \rfloor$. This method was generalized to arbitrary curves by Skorobogatov and Vladut \cite{skorobogatov1990decoding}, and independently by Krachkovskii \cite{krachkovskii1988decoding}, then extended by Duursma \cites{duursma1993algebraic,duursma1993decoding} to correct $\lfloor (\delta_\Gamma -1 )/2 \rfloor - \sigma$ errors, where $\sigma$ is the Clifford defect of the curve \cite{duursma1993decoding}*{Definition 3.7} (is approximately $g/4$). In 1993, Feng and Rao \cite{feng1993decoding} gave a majority voting scheme allowing a decoding up to $\lfloor (\delta_\Gamma -1)/2 \rfloor$ errors. Duursma generalized this result to all AG codes \cite{duursma1993majority}. An efficient algorithm was described by Sakata, Justesen, Madelung, Jensen and H\o hold in \cite{sakata1995fast} using a multidimensional generalization of Massey-Berlekamp algorithm done by Sakata \cite{sakata1990extension}. Kirfel and Pellikaan \cite{kirfel1995minimum} noticed that one can decode beyond $\lfloor (\delta_\Gamma - 1)/2\rfloor$ errors for $1$--point AG codes by studying the Weierstrass semigroup. The reader can refer to \cites{hoholdt1995decoding,hoholdt1998algebraic,pellikaan1993efficient} for more details on decoding methods.

\subsection{Hermitian codes}

The classes of AG codes we study in this paper are defined over the Hermitian curve. Let $\mathbb{F}_q$ be a finite field and define the Hermitian curve $\mathscr{H}_q$ by the affine equation $Y^q+Y=X^{q+1}$. Notice that $\mathscr{H}_q$ is defined over $\mathbb{F}_{q^2}$, that is, its rational points are points of the projective plane $PG(2,q^2)$, satisfying the homogeneous equation $Y^qZ+YZ^q=X^{q+1}$. With respect to the line $Z=0$ at infinity, $\mathscr{H}_q$ has one infinite point $P_\infty=(0:1:0)$ and $q^3$ affine rational points $P_1,\ldots,P_{q^3}$. As usual, we also look at the curve $\mathscr{H}_q$ as the smooth curve defined over the algebraic closure $\bar{\mathbb{F}}_{q^2}$. Then, there is a one-to-one correspondence between the points of $\mathscr{H}_q$ and the places of the function field $\bar{\mathbb{F}}_{q^2}(\mathscr{H}_q)$ of $\mathscr{H}_q$. 

With a Hermitian code we mean a functional AG code of the form $C_L(D,G)$, where the divisor $D$ is defined as the sum $P_1+\cdots+P_{q^3}$ affine rational points of $\mathscr{H}_q$. In our investigations, the divisor $G$ can take two forms. In the \textit{1-point case,} we set $G=sP_\infty$ with integer $s$. In the \textit{degree 3 case,} we put $G=sP$, where $P$ is a place of degree $3$. Let $P_1,P_2,P_3$ be the extensions of $P$ in the constant field extension of $\mathbb F_{q^2}(\mathscr{H}_q)$ of degree $3$. Then $P_1,P_2,P_3$ are degree one places of $\mathbb F_{q^6}(\mathscr{H}_q)$ and, up to labeling the indices, $P_{j+1}=\Frob(P_j)$ where $\Frob$ is the $q^2$-th Frobenius map and the indices are taken modulo $3$. Also, $P$ may be identified with the $\mathbb F_{q^2}$-rational divisor $P_1+P_2+P_3$ of $\mathbb F_{q^6}(\mathscr{H}_q)$. Functional AG codes of the form $C_L(D,sP_\infty)$ and $C_L(D,sP)$ will be called $1$-point Hermitian codes, and Hermitian codes over a degree $3$ place, respectively. In the $1$-point case, the basis of the Riemann-Roch space $\mathscr{L}(sP_\infty)$ can be given explicitly by \cite{stepanov2012codes}:
\begin{equation*}
\mathcal{M} (s):=\left\lbrace x^i y^j  \mid  0\leq i \leq q^2 -1 , 0\leq j \leq q-1, qi+(q+1)j\leq s\right\rbrace.
\end{equation*}
In the degree $3$ case, the basis of
\[\mathscr L(sP)=\left\{\frac{f}{(\ell_1\ell_2\ell_3)^u} \mid f\in \mathbb{F}_{q^2}[X,Y],\deg f\le 3u, v_{P_i}(f) \geq v\right\}\cup\{0\}.\]
can be computed, see \cite{korchmarosnagy2013}. In this formula, $\ell_i=0$ is the equation of the tangent line of $\mathscr{H}_q$ at $P_i$, and $s=u(q+1)-v$, $0\leq v\leq q$. 

The group $\Aut(\mathscr{H}_q)$ of all automorphisms of $\mathscr{H}_q$ is defined over $\mathbb F_{q^2}$. It is a group of projective linear transformations of $PG(2 , q^2 )$, isomorphic to the projective unitary group $PGU( 3, q)$. Furthermore, $\Aut(\mathscr{H}_q)$ acts doubly transitively on the set $\{P_\infty,P_1,\ldots,P_{q^3}\}$ of $\mathbb{F}_{q^2}$-rational points. As it was pointed out in \cite{korchmarosnagy2013}, the automorphism group of $\mathscr{H}_q$ acts transitively on the set of degree $3$ places of $\mathbb F_{q^2}(\mathscr{H}_q)$, as well. Hence, the geometry of a degree $3$ place is independent on the choice of $P$. However, the stabilizer $G_P$ of $P$ in $\Aut(\mathscr{H}_q)$ is not transitive on the set of $q^3+1$ rational points. In fact, $G_P$ is a cyclic group of order $q^2-q+1$ and the number of $G_P$-orbits on the set of rational points is $q+1$. (See \cites{cossidente1999covered, korchmarosnagy2013}.)

\section{Moments of the extended rate of subfield subcodes}
\label{sec:moments}

In order to make our notation consistent, we make the following conventions. Let $\mathcal{X}$ be an algebraic curve over $\mathbb{F}_q$ and $D,G$ effective divisors such that the AG code $C_L(D,G)$ is well defined. Assume that the objects $\delta$ and $\gamma$ determine the curve $\mathcal{X}$ and the divisors $D,G$ in a unique way. Let $s$ be an integer and $\mathbb{F}_r$ be a subfield of $\mathbb{F}_q$. Then, 
\[C_{\delta,r}^\gamma(s) = C_L(D,sG)|_{\mathbb{F}_r}\]
denotes the $\mathbb{F}_q/\mathbb{F}_r$ subfield subcode of the AG code $C_L(D,sG)$. The length of $C_{\delta,r}^\gamma(s)$ is $n=\deg(D)$. 

For the integer $s$, let 
\[R(s)=R_{\delta,r}^\gamma(s)=\frac{\dim_{\mathbb{F}_r} C_{\delta,r}^\gamma(s)}{n}\]
denote the rate of the subfield subcode $C_{\delta,r}^\gamma(s)$.  We extend $R_{\delta,r}^\gamma$ to $\mathbb{R}$ in the usual way: $R_{\delta,r}^\gamma(x)=R_{\delta,r}^\gamma(\lfloor x \rfloor)$.

 
\begin{lemma} \label{lm:R_prop}
Let $g$ be the genus of $\mathcal{X}$ and define 
\[\alpha=\left\lceil\frac{n+2g-2}{\deg(G)}\right\rceil.\]
Then $R(x)$ is a monotone increasing function, with
\[R(x) = \begin{cases}
0 & \text{for $x<0$},\\
1 & \text{for $x\geq \alpha$}.
\end{cases}
\]
\end{lemma}
\begin{proof}
If $s\deg(G)>n+2g-2$, then $\deg(D+W-G)<0$, and
\[C_\Omega(D,G)\cong C_L(D,D+W-G) = \{0\}.\]
Hence, if $s\geq \alpha$, then $C_L(D,sG)=\mathbb{F}_q^n$ and $C_L(D,sG)|_{\mathbb{F}_r}=\mathbb{F}_r^n$. 
\end{proof}
The following observation has been made in \cite{khalfaoui2019dimension}*{Theorem 5.1} for the special case of a one point divisor of a Hermitian curve.
\begin{lemma}
For $0\leq x < n/(r\deg(G))$, we have $R(x)=1/n$. 
\end{lemma}
\begin{proof}
As the divisor $sG$ is positive for $s>0$, the constant vectors are in $C_L(D,sG)|_{\mathbb{F}_r}$ and $R(s)\geq 1/n$ holds. Assume $R(s)>1/n$, that is, the subfield subcode contains a non constant element $\bm{v}=(f(P_1),\ldots,f(P_n))$ with $f\in \mathscr{L}(sG)$. Since $f$ cannot have more than $\deg(sG)$ zeros, $\bm{v}$ cannot have the same entry more than $s\deg(G)$ times. This implies $r\deg(sG)\geq n$. 
\end{proof}

Lemma \ref{lm:R_prop} implies that we can consider $R(x)$ as the distribution function of some random variable $\xi$, cf. \cite{ShProb1}*{Definition 1, Section 2.3}. 

\begin{lemma} \label{lm:moments}
Let $R(x)$ be the extended rate function of a class of subfield subcodes $C_L(D,sG)|_{\mathbb{F}_r}$. Define the integer $\alpha$ as in Lemma \ref{lm:R_prop}. Let $\xi$ be a random variable with distribution function $R(x)$. Then
\[\E(\xi)=\sum_{s=0}^\alpha 1-R(s), \qquad \E(\xi^2)=\sum_{s=0}^\alpha (2s+1)(1-R(s)).\]
\end{lemma}
\begin{proof}
This follows from \cite{ShProb1}*{Section 2.6, Corollary 2}.
\end{proof}

\begin{remark}
Considered as a distribution function, $R_{\delta,r}^\gamma(s)$ has an expectation $\E_{\delta,r}^\gamma$, a variance $\Var_{\delta,r}^\gamma$ and a standard deviation $\D_{\delta,r}^\gamma$. These constants can be computed from the true dimensions of the subfield subcodes using Lemma \ref{lm:moments} and the well known formulas of random variables.
\end{remark}

\section{Computed true dimensions of Hermitian subfield subcodes}
\label{sec:subcodes}

Let $q$ be a prime power. We say that the object $\delta=q$ determines the Hermitian curve $\mathscr{H}_q$ over $\mathbb{F}_{q^2}$, together with the divisor $D$ which is the sum of affine rational points of $\mathscr{H}_q$. The objects $\gamma=\text{1-pt}$ or $\gamma=\text{deg-3}$ determine the divisor $G$ to be equal either to the rational infinite place $P_\infty$, or the degree 3 Hermitian place $P$, respectively. That being said, for any integer $s$ and subfield $\mathbb{F}_r$ of $\mathbb{F}_{q^2}$, the Hermitian subfield subcodes 
\[C^\text{1-pt}_{q,r}(s) =C_L(D,sP_\infty)|_{\mathbb{F}_{r}}, \qquad
C^\text{deg-3}_{q,r}(s) = C_L(D,sP)|_{\mathbb{F}_{r}}
\]
are well defined and consistent with the notation of section \ref{sec:moments}. These codes are $\mathbb{F}_r$-linear codes of length $n=q^3$. 

%

Let $R^\text{1-pt}_{q,r}(s)$ and $R^\text{deg-3}_{q,r}(s)$ be the true rates of the codes $C^\text{1-pt}_{q,r}(s)$ and $C^\text{deg-3}_{q,r}(s)$. Using the GAP \cite{GAP4.10.2} package \texttt{HERmitian} \cite{HERmitian01}, we have been able to compute the true dimension values of the codes $C_{q,q}^\text{1-pt}(s)$, $C_{q,q}^\text{deg-3}(s)$ for
\[q \in \{2,3,4,5,7,8,9,11,13\}\]
and the binary codes $C_{q,2}^\text{1-pt}(s)$, $C_{q,2}^\text{deg-3}(s)$ for
\[q\in \{2,4,8,16\}.\]
(Cf. \cite{khalfaoui2019dimension} for preliminary results on explicit computation of subfield subcodes of Hermitian $1$-point codes.) 

As given in Lemma \ref{lm:moments}, we computed the expectations $\E_{q,q}^\text{1-pt}$, $\E_{q,2}^\text{1-pt}$, $\E_{q,q}^\text{deg-3}$, $\E_{q,2}^\text{deg-3}$, the variances $\Var_{q,q}^\text{1-pt}$, $\Var_{q,2}^\text{1-pt}$, $\Var_{q,q}^\text{deg-3}$, $\Var_{q,2}^\text{deg-3}$, and the standard deviations $\D^\textbf{1-pt}_{q,r}$, $\D_{q,2}^\text{1-pt}$, $\D_{q,q}^\text{deg-3}$, $\D_{q,2}^\text{deg-3}$ for these true rates. The numerical results are shown in Table \ref{tbl:evar_q} for $q=3,4,5,7,8,9,11,13$ and $r=q$, and in Table \ref{tbl:evar_2} for $q=2,4,8,16$ and $r=2$. In Figure \ref{fig:ratios}, we present the ratios $\E^\gamma_{q,r}\deg(G)/n$ and $\D^\gamma_{q,r}\deg(G)/n$, where $\gamma\in\{\text{1-pt, deg-3}\}$. While our data sets are small, these figures motivate the following open problem.
\begin{problem}
Are there constants $c_1,c_2>0$ such that
\[\E^\text{1-pt}_{q,q} \approx \E^\text{deg-3}_{q,q} \approx c_1 q^3/\deg(G), \qquad \D^\text{1-pt}_{q,q} \approx \D^\text{deg-3}_{q,q} \approx c_2 q^3/\deg(G),\]
where $a\approx b$ means $a/b \to 1$ with $q\to \infty$. 
\end{problem}
\begin{remark} \label{rem:choice}
Our data suggests that for small $q$, the choice $c_1=0.75$ and $c_2=0.2$ is sound. 
\end{remark}

\begin{table}
\caption{Expectations and variances for Hermitian $\mathbb{F}_{q^2}/\mathbb{F}_{q}$ subfield subcodes \label{tbl:evar_q}}
\centering
\begin{tabular}{|c| %
>{\raggedleft}p{0.15\textwidth}>{\raggedleft}p{0.15\textwidth}| %
>{\raggedleft}p{0.15\textwidth}>{\raggedleft\arraybackslash}p{0.15\textwidth}|} 
\hline
\multirow{2}{*}{$q$} & \multicolumn{2}{c|}{1-point codes} & \multicolumn{2}{c|}{Codes over a degree $3$ place} \\ 
\cline{2-5}
 & \multicolumn{1}{c}{Expectation} & \multicolumn{1}{c|}{Variance} & \multicolumn{1}{c}{Expectation} & \multicolumn{1}{c|}{Variance} \\ 
\hline
$3$ & $20.15$ & $53.46$ & $7.63$ & $4.09$ \\
$4$ & $48.66$ & $246.79$ & $17.77$ & $16.02$ \\
$5$ & $95.04$ & $841.16$ & $33.37$ & $60.18$ \\
$7$ & $259.10$ & $5\,553.32$ & $88.99$ & $503.78$ \\
$8$ & $385.49$ & $11\,862.84$ & $131.61$ & $1\,106.63$ \\
$9$ & $546.30$ & $23\,541.65$ & $186.22$ & $2\,206.21$ \\
$11$ & $992.73$ & $74\,679.83$ & $336.49$ & $7\,262.13$ \\
$13$ & $1\,631.29$ & $197\,675.07$ & $550.94$ & $19\,807.94$ \\
\hline
\end{tabular}
\end{table}

\begin{table}
\caption{Expectations and variances for Hermitian $\mathbb{F}_{q^2}/\mathbb{F}_{2}$ subfield subcodes \label{tbl:evar_2}}
\centering
\begin{tabular}{|c| %
>{\raggedleft}p{0.15\textwidth}>{\raggedleft}p{0.15\textwidth}| %
>{\raggedleft}p{0.15\textwidth}>{\raggedleft\arraybackslash}p{0.15\textwidth}|} 
\hline
\multirow{2}{*}{$q$} & \multicolumn{2}{c|}{1-point codes} & \multicolumn{2}{c|}{Codes over a degree $3$ place} \\ 
\cline{2-5}
 & \multicolumn{1}{c}{Expectation} & \multicolumn{1}{c|}{Variance} & \multicolumn{1}{c}{Expectation} & \multicolumn{1}{c|}{Variance} \\ 
\hline
$2$ & $5.38$ & $6.48$ & $2.12$ & $0.86$ \\
$4$ & $54.86$ & $164.96$ & $20.38$ & $10.52$ \\
$8$ & $458.22$ & $4\,838.52$ & $162.50$ & $216.32$ \\
$16$ & $3\,698.92$ & $195\,390.48$ & $1\,303.40$ & $6\,029.44$ \\
\hline
\end{tabular}
\end{table}

\begin{figure}
\caption{The ratios of expectations and standard deviations to $n/\deg(G)$\label{fig:ratios}}
\centering
\begin{tikzpicture}[scale=0.7]
\begin{axis}[xlabel={$q\in\{2,4,5,7,8,9,11,13\}$, $r=q$}]
\addplot[color=blue,mark=x] coordinates {
(3,0.746) (4,0.760) (5,0.760) (7,0.755) (8,0.753) (9,0.749) (11,0.746) (13,0.743)
};
\addplot[color=blue,mark=o] coordinates {
(3,0.848) (4,0.833) (5,0.801) (7,0.778) (8,0.771) (9,0.766) (11,0.758) (13,0.752)
};
\addplot[dashed,color=red,mark=x] coordinates {
(3,0.271) (4,0.245) (5,0.232) (7,0.217) (8,0.213) (9,0.210) (11,0.205) (13,0.202) 
};
\addplot[dashed,color=red,mark=o] coordinates {
(3,0.225) (4,0.188) (5,0.186) (7,0.196) (8,0.195) (9,0.193) (11,0.192) (13,0.192) 
};
\end{axis}
%
\begin{axis}[xlabel={$q\in\{2,4,8,16\}$, $r=2$},legend style={at={(0.97,0.5)},anchor=east},xshift=8cm]
\addplot[color=blue,mark=x] coordinates {
(2,0.672) (4,0.857) (8,0.895) (16,0.903) 
};
\addlegendentry{$E^\textbf{1-pt}_{q,r}$}
\addplot[color=blue,mark=o] coordinates {
(2,0.797) (4,0.955) (8,0.952) (16,0.955)
};
\addlegendentry{$E^\textbf{deg-3}_{q,r}$}
\addplot[dashed,color=red,mark=x] coordinates {
(2,0.318) (4,0.201) (8,0.136) (16,0.108) 
};
\addlegendentry{$D^\textbf{1-pt}_{q,r}$}
\addplot[dashed,color=red,mark=o] coordinates {
(2,0.348) (4,0.152) (8,0.086) (16,0.057)
};
\addlegendentry{$D^\textbf{deg-3}_{q,r}$}
\end{axis}
\end{tikzpicture}
\end{figure}
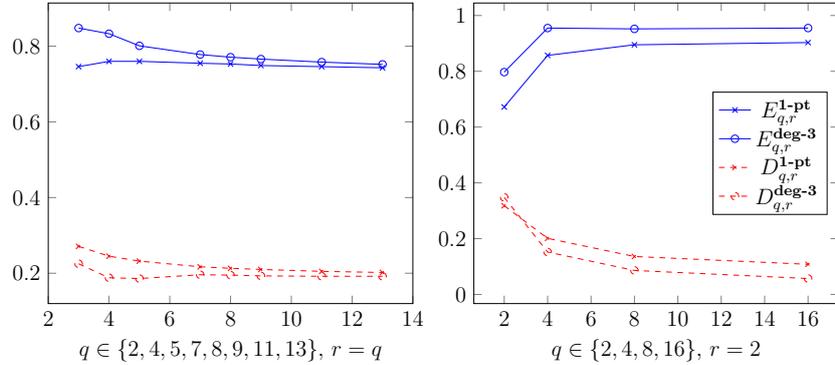

\section{Distribution fitting}
\label{sec:fitting}

In general, no explicit formula is known for the true dimension of subfield subcodes of AG codes. Our goal is to use the method of distribution fitting in order to study the behavior of these true dimensions in the case of subfield subcodes of Hermitian codes. 

As in the previous sections, we use the notation $\mathscr{H}_q$ for the Hermitian curve over $\mathbb{F}_{q^2}$, $P_\infty,P$ for the places of degree $1$ and $3$, $D$ and $G\in\{P_\infty,P\}$ for the divisors, and $C_{q,r}^\gamma(s)$, $\gamma \in \{\text{1-pt, deg-3}\}$, for the  $\mathbb{F}_{q^2}/\mathbb{F}_r$ subfield subcodes $C_L(D,sG)|_{\mathbb{F}_{r}}$. Then, with fixed $q,r$ and $\gamma \in \{\text{1-pt, deg-3}\}$ the dimensions of the subfield subcodes are given by the extended rate function $R_{q,r}^\gamma(x)$.

\[R_{q,q}^\text{1-pt}(x), \quad R_{q,2}^\text{1-pt}(x), \quad R_{q,q}^\text{deg-3}(x), \quad R_{q,2}^\text{deg-3}(x).\]
Our goal is to consider these functions as distribution functions and fit some well known probability distribution functions to our experimental rate function $R(x)$. 

We obtain numerical results by using the distribution fitting methods offered by MATLAB's Statistics and Machine Learning Toolbox \cite{matlabstat}. The technique MLE (Maximum Likelihood Estimation) is a method for estimating the parameters of a probability distribution from a data set. The method finds the parameter values maximizing the logarithm of the likelihood function \cite{eliason1993maximum}. In order to compare different distributions for a given data set, one can use the log-likelihood values for a ranking. This is implemented MATLAB's \texttt{fitmethis} function \cite{fitmethis2020matlab}. Notice that \texttt{fitmethis} also computes the AIC value for each distribution, which stands for Akaike Information Criterion, that measures the quality of a model (distribution) versus the other models. It has the formula 
\[AIC=2l-2\log(\hat{L})\]
where $l$ is the number of parameters and $\hat{L}$ is the maximum values of the likelihood function. In the case of AIC, smaller values correspond to better fitting distributions (see \cite{konishi2008information}).  

In our comparisons, we restricted ourselves to parametric distributions having at most two parameters, that is, we used MATLAB's \texttt{fitmethis} to compare the log-likelihood values of the following distributions: normal, exponential, gamma, logistic, uniform, extreme value, Rayleigh, beta, Nakagami, Rician, inverse Gaussian, Birnbaum-Saunders, log-logistic, log-normal and Weibull. We can summarize the results as follows:

\begin{proposition} \label{pr:main}
\begin{enumerate}
\item The best fitting distribution was the extreme value distribution for $R_{q,q}^\text{1-pt}(x)$, $q\in \{4,5,7,8,9,11,13\}$, for $R_{q,q}^\text{deg-3}(x)$, $q\in \{7,8,9,11,13\}$, and for $R_{8,2}^\text{1-pt}(x)$, $R_{16,2}^\text{1-pt}(x)$, $R_{4,2}^\text{deg-3}(x)$, $R_{8,2}^\text{deg-3}(x)$, and $R_{16,2}^\text{deg-3}(x)$.
\item For the missing cases $R_{2,2}^\text{1-pt}(x)$, $R_{3,3}^\text{1-pt}(x)$, $R_{2,2}^\text{deg-3}(x)$, $R_{3,3}^\text{deg-3}(x)$, $R_{4,4}^\text{deg-3}(x)$, and $R_{5,5}^\text{deg-3}(x)$, the best fitting distribution was the gamma distribution. 
\item The second best fitting distribution was the extreme value distribution for $R_{3,3}^\text{1-pt}(x)$, $R_{3,3}^\text{deg-3}(x)$, $R_{4,4}^\text{deg-3}(x)$, $R_{5,5}^\text{deg-3}(x)$. 
\end{enumerate}
\end{proposition}


Our results show that for $q\geq 3$, among the two-parameter distributions, the extreme value distribution function is a good estimation of the rate function of subfield subcodes of Hermitian codes. 

The extreme value distribution is also referred to as Gumbel or type 1 Fisher-Tippet distribution. In probability theory, these are the limiting distributions of the minimum of a large number of unbounded identically distributed random variables. The extreme value distribution function is
\[F(x;\alpha,\beta)=1-\exp\left(-\exp\left(\frac{x-\alpha}{\beta}\right)\right),\]
with location parameter $\alpha\in \mathbb{R}$ and a scale parameter $\beta>0$. The mean $\mu$ and the variance $\sigma^2$ are
\[\mu=\alpha+\beta \gamma, \qquad \sigma^2=\frac{\pi^2}{6}\beta^2,\]
where 
\[\gamma = \int_1^\infty \left(-\frac{1}{x}+\frac{1}{\lfloor x \rfloor}\right) dx \approx 0.57721566490153 \]
is the Euler-Mascheroni constant, see \cite{extreme2000kotz}*{Section 1.4}. With given empirical mean and variance of the data series, the parameters can be computed by
\[\alpha=\mu - \frac{\sqrt{6}\gamma}{\pi} \sigma, \qquad \beta = \frac{\sqrt{6}}{\pi}\sigma.\]

\begin{figure}
\caption{Estimating the extended rate function by extreme value distribution for subfield subcodes Hermitian codes\label{fig:pics_qisr}}
\centering
\includegraphics[width=0.45\columnwidth]{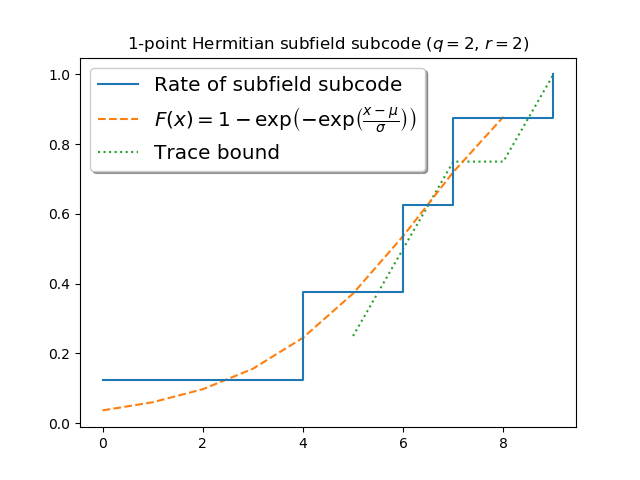}
\includegraphics[width=0.45\columnwidth]{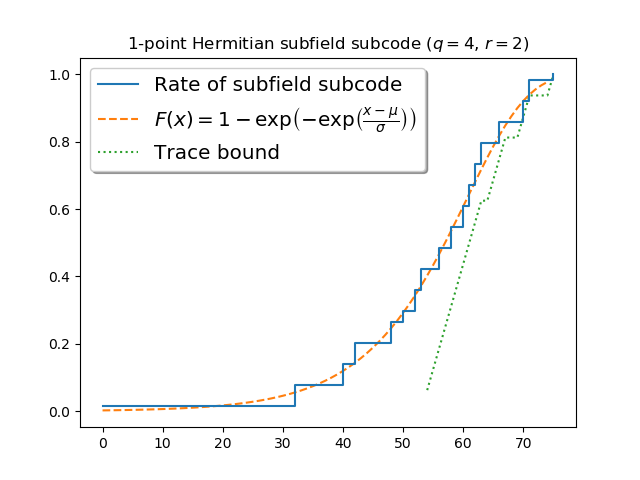}
\includegraphics[width=0.45\columnwidth]{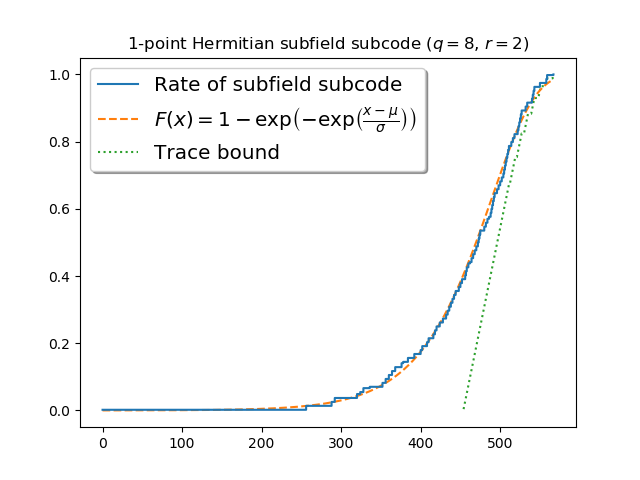}
\includegraphics[width=0.45\columnwidth]{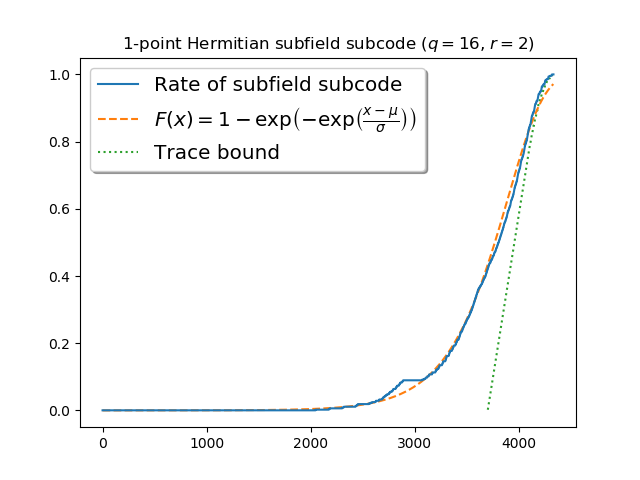}
%
\includegraphics[width=0.45\columnwidth]{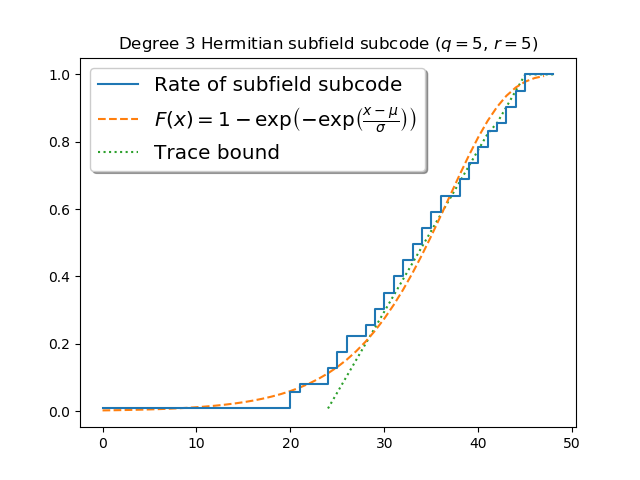}
\includegraphics[width=0.45\columnwidth]{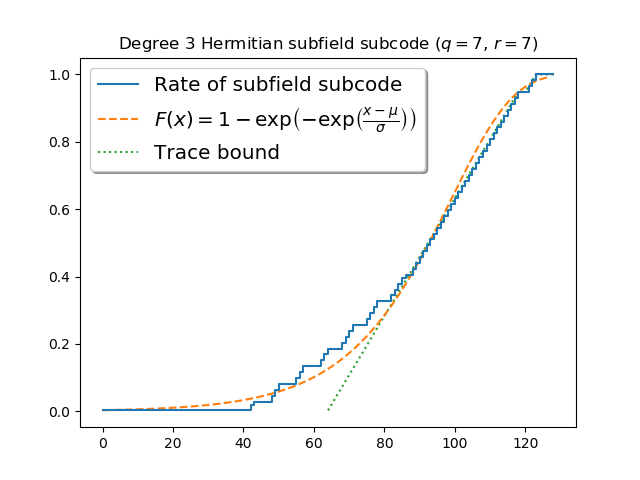}
\includegraphics[width=0.45\columnwidth]{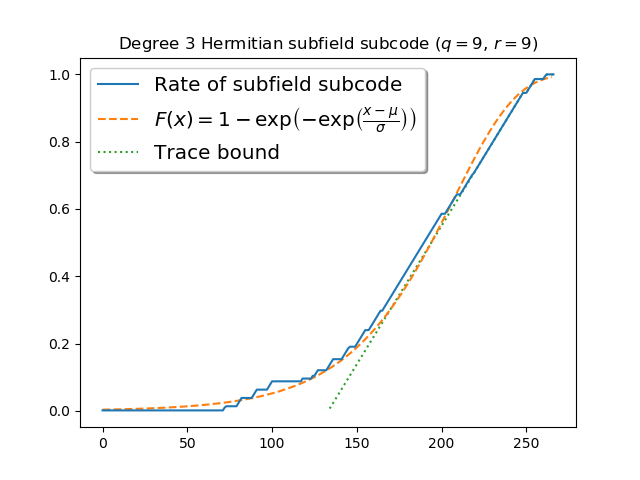}
\includegraphics[width=0.45\columnwidth]{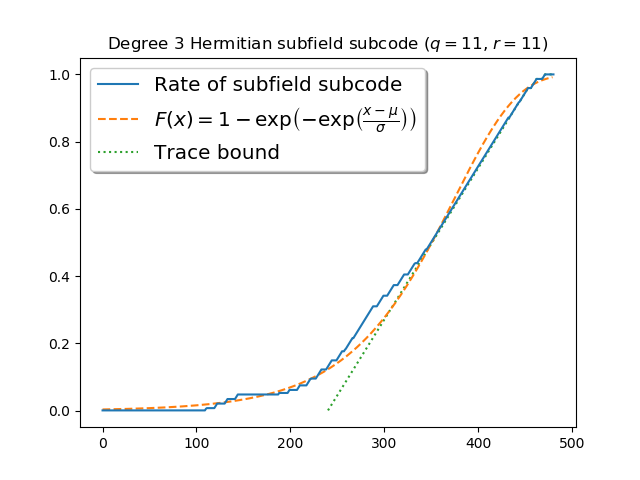}
\end{figure}

In Figure \ref{fig:pics_qisr} we visualized the fitting of the extreme value distribution function to our experimental results on the true dimension of subfield subcodes of Hermitian codes. 

The occurrence of the extreme value distribution in the context of subfield subcodes of AG codes may be somewhat surprising and we cannot give a plain mathematical explanation for this. However, the rank of random matrices over finite fields is known to be related to the class of Gumbel type distributions, see Cooper's result \cite{cooper2000randmat}*{Theorem 5} for the theoretical background. This theory has been applied to parameter estimates of random erasure codes by Studholme and Blake \cite{randmat2010erasure}.

\section{Application: Estimating the key size of McEliece Cryptosystem}
\label{sec:application}

The largest (but not the only) part of the public key of the McEliece cryptosystem is the matrix $A$ which defines the underlying error correction code. $A$ is either the $n\times k$ generator matrix, or the $n\times (n-k)$ parity check matrix. In either case, $A$ may be assumed to be in standard form, which means that the public key is given by $k(n-k)$ elements of $\mathbb{F}_r$. Hence, the key size is
\[\log_2(r)k(n-k).\] 
Hence, for a fixed field $\mathbb{F}_r$ and length $n$, the key size is propotional to $R(1-R)$, see \cite{Niebuhr2012}. The true values of $R^\gamma_{q,r}(s)(1-R^\gamma_{q,r}(s))$ can be estimated by $F(x)(1-F(x))$, where $F(x)$ is the extreme value distribution function, see Figure \ref{fig:keysize}.

\begin{figure}
\caption{Estimating the key size $n^2 R(1-R)$}
\label{fig:keysize}
\centering
\includegraphics[width=0.4\columnwidth]{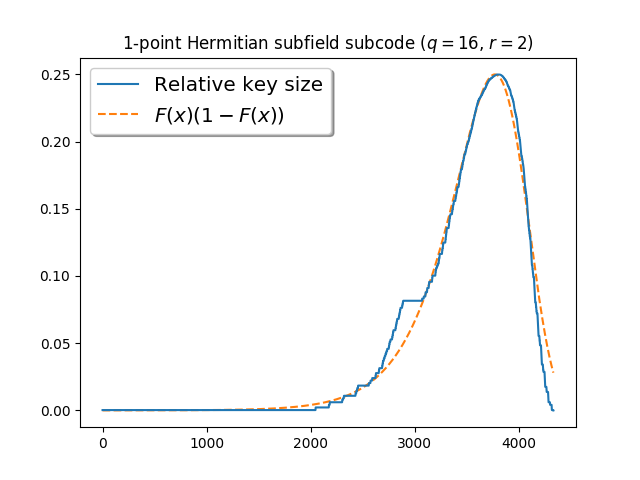}
\includegraphics[width=0.4\columnwidth]{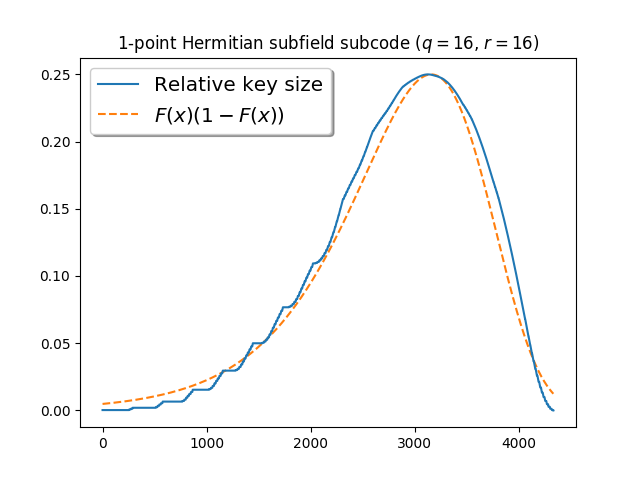}
\end{figure}

\section{Conclusion and future work}
\label{sec:conclusion}

The main goal of this study was to establish an approximating formula of the true dimension of the subfield subcodes of Hermitian codes. We conducted an experimental study to analyze the datasets of the true dimension of the $\mathbb{F}_r$-linear codes $C^{1-pt}_{q,r}(s)$, $C^{deg-3}_{q,r}(s)$ for $q\in \left\lbrace 2,3,4,5,7,8,9,11,13,16 \right\rbrace$, $r=2$ or $r=q$, and $s$ is an integer parameter running from $0$ to $q^3+(q+1)(q-2)$. This analysis helped us to derive new properties of their structure and led to an approach that might be useful for further research and applications.

Theoretically, the main contribution of this work is the set up of statistical formulas such as moment and expectation by mean of the extended rate function of the underlying classes of subfield subcodes of Hermitian codes.

From a statistical perspective, the main result is the comparison of the fitting of our datasets of true dimensions to well known distribution functions of MATLAB's Statistics and Machine Learning Toolbox, using the method of \texttt{fitmethis}. 

We found that the extreme value distribution is the best fitting one for $q>5$ and the second best fitting distribution for smalles values of $q$. Also the gamma and the normal distributions have good fitting properties. Our proposal is to use the extreme value distribution function to estimate the true dimension of subfield subcodes of Hermitian codes. In the last section of this paper, we applied this formula to give an approximation for the key size of the McEliece scheme, depending on the parameter $s$.

In the future, we aim to replace Goppa codes in McEliece's original version with a family of codes that permit to reduce the public key size and to increase the code rate by maintaining a given level of security. Therefore, we intend to analyze McEliece cryptosystems based on subclasses of subfield subcodes of Hermitian codes. Our future work will include experiments, simulations, and security and cryptanalysis of the McEliece scheme in term of its public key size and other parameters. The measurements are based on attacks with supposed lowest complexity, e.g. information set decoding or the Schur product distinguisher. 

\section*{Acknowledgment}


The presented work was carried out within the project ``Security Enhancing Technologies for the Internet of Things'' 2018-1.2.1-NKP-2018-00004, supported by the National Research, Development and Innovation Fund of Hungary, financed under the 2018-1.2.1-NKP funding scheme. Partially supported by NKFIH-OTKA Grants 119687 and 115288.

The authors would like to thank Levente Butty\'an (Budapest University of Technology, Hungary) for motivating discussions and M\'aty\'as Barczy (University of Szeged, Hungary) for his help to deal successfully with the concepts from probability theory and statistics. 
\medskip
\bibliographystyle{alpha}

\begin{bibdiv}
\begin{biblist}
\bib{nnn}{book}{
      author={Alagic, Gorjan},
      author={Alperin-Sheriff, Jacob},
      author={Apon, Daniel},
      author={Cooper, David},
      author={Dang, Quynh},
      author={Liu, Yi-Kai},
      author={Miller, Carl},
      author={Moody, Dustin},
      author={others},
       title={Status report on the first round of the {NIST} post-quantum
  cryptography standardization process},
   publisher={US Department of Commerce, National Institute of Standards and
  Technology},
        date={2019},
         url={https://doi.org/10.6028/NIST.IR.8240},
}

\bib{google2019quantum}{article}{
      author={Arute, Frank},
      author={Arya, Kunal},
      author={Babbush, Ryan},
      author={Bacon, Dave},
      author={Bardin, Joseph~C.},
      author={Barends, Rami},
      author={Biswas, Rupak},
      author={Boixo, Sergio},
      author={Brandao, Fernando G. S.~L.},
      author={Buell, David~A.},
      author={Burkett, Brian},
      author={Chen, Yu},
      author={Chen, Zijun},
      author={Chiaro, Ben},
      author={Collins, Roberto},
      author={Courtney, William},
      author={Dunsworth, Andrew},
      author={Farhi, Edward},
      author={Foxen, Brooks},
      author={Fowler, Austin},
      author={Gidney, Craig},
      author={Giustina, Marissa},
      author={Graff, Rob},
      author={Guerin, Keith},
      author={Habegger, Steve},
      author={Harrigan, Matthew~P.},
      author={Hartmann, Michael~J.},
      author={Ho, Alan},
      author={Hoffmann, Markus},
      author={Huang, Trent},
      author={Humble, Travis~S.},
      author={Isakov, Sergei~V.},
      author={Jeffrey, Evan},
      author={Jiang, Zhang},
      author={Kafri, Dvir},
      author={Kechedzhi, Kostyantyn},
      author={Kelly, Julian},
      author={Klimov, Paul~V.},
      author={Knysh, Sergey},
      author={Korotkov, Alexander},
      author={Kostritsa, Fedor},
      author={Landhuis, David},
      author={Lindmark, Mike},
      author={Lucero, Erik},
      author={Lyakh, Dmitry},
      author={Mandr{\`a}, Salvatore},
      author={McClean, Jarrod~R.},
      author={McEwen, Matthew},
      author={Megrant, Anthony},
      author={Mi, Xiao},
      author={Michielsen, Kristel},
      author={Mohseni, Masoud},
      author={Mutus, Josh},
      author={Naaman, Ofer},
      author={Neeley, Matthew},
      author={Neill, Charles},
      author={Niu, Murphy~Yuezhen},
      author={Ostby, Eric},
      author={Petukhov, Andre},
      author={Platt, John~C.},
      author={Quintana, Chris},
      author={Rieffel, Eleanor~G.},
      author={Roushan, Pedram},
      author={Rubin, Nicholas~C.},
      author={Sank, Daniel},
      author={Satzinger, Kevin~J.},
      author={Smelyanskiy, Vadim},
      author={Sung, Kevin~J.},
      author={Trevithick, Matthew~D.},
      author={Vainsencher, Amit},
      author={Villalonga, Benjamin},
      author={White, Theodore},
      author={Yao, Z.~Jamie},
      author={Yeh, Ping},
      author={Zalcman, Adam},
      author={Neven, Hartmut},
      author={Martinis, John~M.},
       title={Quantum supremacy using a programmable superconducting
  processor},
        date={2019},
        ISSN={1476-4687},
     journal={Nature},
      volume={574},
      number={7779},
       pages={505\ndash 510},
         url={https://doi.org/10.1038/s41586-019-1666-5},
}

\bib{baldi2013security}{article}{
      author={Baldi, Marco},
      author={Bianchi, Marco},
      author={Chiaraluce, Franco},
       title={Security and complexity of the mceliece cryptosystem based on
  quasi-cyclic low-density parity-check codes},
        date={2013},
     journal={IET Information Security},
      volume={7},
      number={3},
       pages={212\ndash 220},
}

\bib{cooper2000randmat}{inproceedings}{
      author={Cooper, C.},
       title={On the distribution of rank of a random matrix over a finite
  field},
        date={2000},
   booktitle={Proceedings of the {N}inth {I}nternational {C}onference
  ``{R}andom {S}tructures and {A}lgorithms'' ({P}oznan, 1999)},
      volume={17},
       pages={197\ndash 212},
  url={https://doi.org/10.1002/1098-2418(200010/12)17:3/4<197::AID-RSA2>3.3.CO;2-B},
      review={\MR{1801132}},
}

\bib{cossidente1999covered}{article}{
      author={Cossidente, Antonio},
      author={Korchm\'{a}ros, Gabor},
      author={Torres, Fernando},
       title={On curves covered by the {H}ermitian curve},
        date={1999},
        ISSN={0021-8693},
     journal={J. Algebra},
      volume={216},
      number={1},
       pages={56\ndash 76},
         url={https://doi.org/10.1006/jabr.1998.7768},
      review={\MR{1694594}},
}

\bib{Couvreur2017}{article}{
      author={Couvreur, Alain},
      author={M\'{a}rquez-Corbella, Irene},
      author={Pellikaan, Ruud},
       title={Cryptanalysis of {M}c{E}liece cryptosystem based on algebraic
  geometry codes and their subcodes},
        date={2017},
        ISSN={0018-9448},
     journal={IEEE Trans. Inform. Theory},
      volume={63},
      number={8},
       pages={5404\ndash 5418},
         url={https://doi.org/10.1109/TIT.2017.2712636},
      review={\MR{3683571}},
}

\bib{fitmethis2020matlab}{misc}{
      author={de~Castro, Francisco},
       title={{fitmethis}, {V}ersion 1.3.0.0},
        date={2020},
  url={https://www.mathworks.com/matlabcentral/fileexchange/40167-fitmethis},
        note={MATLAB Central File Exchange},
}

\bib{Del}{article}{
      author={Delsarte, Philippe},
       title={On subfield subcodes of modified {R}eed-{S}olomon codes},
        date={1975},
        ISSN={0018-9448},
     journal={IEEE Trans. Information Theory},
      volume={IT-21},
      number={5},
       pages={575\ndash 576},
      review={\MR{0411819}},
}

\bib{driencourt1985some}{inproceedings}{
      author={Driencourt, Yves},
       title={Some properties of elliptic codes over a field of characteristic
  2},
organization={Springer},
        date={1985},
   booktitle={International conference on applied algebra, algebraic
  algorithms, and error-correcting codes},
       pages={185\ndash 193},
}

\bib{duursma1993algebraic}{article}{
      author={Duursma, Iwan~M.},
       title={Algebraic decoding using special divisors},
        date={1993},
     journal={IEEE transactions on information theory},
      volume={39},
      number={2},
       pages={694\ndash 698},
}

\bib{duursma1993decoding}{article}{
      author={Duursma, Iwan~M.},
       title={Decoding--codes from curves and cyclic codes},
        date={1993},
}

\bib{duursma1993majority}{article}{
      author={Duursma, Iwan~M.},
       title={Majority coset decoding},
        date={1993},
     journal={IEEE transactions on information theory},
      volume={39},
      number={3},
       pages={1067\ndash 1070},
}

\bib{khalfaoui2019dimension}{article}{
      author={El~Khalfaoui, Sabira},
      author={Nagy, G\'abor~P.},
       title={On the dimension of the subfield subcodes of 1-point {H}ermitian
  codes},
        date={2019},
        ISSN={1930-5346},
     journal={Advances in Mathematics of Communications},
      volume={0},
      number={0},
       pages={0},
      eprint={arxiv:1906.10444},
  url={http://aimsciences.org//article/id/3c07c763-7026-4b14-a786-8038c3bad13b},
}

\bib{eliason1993maximum}{book}{
      author={Eliason, Scott~R},
       title={Maximum likelihood estimation: Logic and practice},
   publisher={Sage},
        date={1993},
      number={96},
}

\bib{feng1993decoding}{article}{
      author={Feng, G.-L.},
      author={Rao, Thammavarapu R.~N.},
       title={Decoding algebraic-geometric codes up to the designed minimum
  distance},
        date={1993},
     journal={IEEE Transactions on Information Theory},
      volume={39},
      number={1},
       pages={37\ndash 45},
}

\bib{GAP4.10.2}{misc}{
       title={{GAP} {\textendash} {G}roups, {A}lgorithms, and {P}rogramming,
  {V}ersion 4.10.2},
organization={The GAP {G}roup},
        date={2019},
         url={https://www.gap-system.org},
}

\bib{hirschfeld2008algebraic}{book}{
      author={Hirschfeld, J. W.~P.},
      author={Korchm\'{a}ros, G.},
      author={Torres, F.},
       title={Algebraic curves over a finite field},
      series={Princeton Series in Applied Mathematics},
   publisher={Princeton University Press, Princeton, NJ},
        date={2008},
        ISBN={978-0-691-09679-7},
      review={\MR{2386879}},
}

\bib{hoholdt1995decoding}{article}{
      author={Hoholdt, Tom},
      author={Pellikaan, Ruud},
       title={On the decoding of algebraic-geometric codes},
        date={1995},
     journal={IEEE Transactions on Information Theory},
      volume={41},
      number={6},
       pages={1589\ndash 1614},
}

\bib{hoholdt1998algebraic}{article}{
      author={H{\o}holdt, Tom},
      author={Van~Lint, Jacobus~H.},
      author={Pellikaan, Ruud},
       title={Algebraic geometry codes},
        date={1998},
     journal={Handbook of coding theory},
      volume={1},
      number={Part 1},
       pages={871\ndash 961},
}

\bib{justesen1989construction}{article}{
      author={Justesen, J{\o}rn},
      author={Larsen, Knud~J.},
      author={Jensen, Helge~Elbr{\o}nd},
      author={Havemose, Allan},
      author={Hoholdt, Tom},
       title={Construction and decoding of a class of algebraic geometry
  codes},
        date={1989},
     journal={IEEE Transactions on Information Theory},
      volume={35},
      number={4},
       pages={811\ndash 821},
}

\bib{justesen1992fast}{article}{
      author={Justesen, J{\o}rn},
      author={Larsen, Knud~J.},
      author={Jensen, Helge~Elbr{\o}nd},
      author={Hoholdt, Tom},
       title={Fast decoding of codes from algebraic plane curves},
        date={1992},
     journal={IEEE Transactions on Information Theory},
      volume={38},
      number={1},
       pages={111\ndash 119},
}

\bib{kirfel1995minimum}{article}{
      author={Kirfel, Christoph},
      author={Pellikaan, Ruud},
       title={The minimum distance of codes in an array coming from telescopic
  semigroups},
        date={1995},
     journal={IEEE Transactions on information theory},
      volume={41},
      number={6},
       pages={1720\ndash 1732},
}

\bib{konishi2008information}{book}{
      author={Konishi, Sadanori},
      author={Kitagawa, Genshiro},
       title={Information criteria and statistical modeling},
   publisher={Springer Science \& Business Media},
        date={2008},
}

\bib{korchmarosnagy2013}{article}{
      author={Korchm\'{a}ros, G\'{a}bor},
      author={Nagy, G\'{a}bor~P.},
       title={Hermitian codes from higher degree places},
        date={2013},
        ISSN={0022-4049},
     journal={J. Pure Appl. Algebra},
      volume={217},
      number={12},
       pages={2371\ndash 2381},
         url={https://doi.org/10.1016/j.jpaa.2013.04.002},
      review={\MR{3057317}},
}

\bib{extreme2000kotz}{book}{
      author={Kotz, Samuel},
      author={Nadarajah, Saralees},
       title={Extreme value distributions},
   publisher={Imperial College Press, London},
        date={2000},
        ISBN={1-86094-224-5},
         url={https://doi.org/10.1142/9781860944024},
        note={Theory and applications},
      review={\MR{1892574}},
}

\bib{krachkovskii1988decoding}{inproceedings}{
      author={Krachkovskii, V.~Yu.},
       title={Decoding of codes on algebraic curves},
        date={1988},
   booktitle={Conference odessa},
}

\bib{loidreau2000strengthening}{inproceedings}{
      author={Loidreau, Pierre},
       title={Strengthening mceliece cryptosystem},
organization={Springer},
        date={2000},
   booktitle={International conference on the theory and application of
  cryptology and information security},
       pages={585\ndash 598},
}

\bib{mceliece1978public}{article}{
      author={McEliece, Robert~J},
       title={A public-key cryptosystem based on algebraic},
        date={1978},
}

\bib{menezes2013applications}{book}{
      author={Menezes, Alfred~J.},
      author={Blake, Ian~F.},
      author={Gao, XuHong},
      author={Mullin, Ronald~C.},
      author={Vanstone, Scott~A.},
      author={Yaghoobian, Tomik},
       title={Applications of finite fields},
   publisher={Springer Science \& Business Media},
        date={2013},
      volume={199},
}

\bib{HERmitian01}{misc}{
      author={Nagy, G\'abor~P.},
      author={El~Khalfaoui, Sabira},
       title={{HERmitian}, {C}omputing with divisors, {R}iemann-{R}och spaces
  and {AG}-odes of {H}ermitian curves, {V}ersion 0.1},
        date={2019},
         url={https://github.com/nagygp/Hermitian},
        note={GAP package},
}

\bib{Niebuhr2012}{article}{
      author={Niebuhr, Robert},
      author={Meziani, Mohammed},
      author={Bulygin, Stanislav},
      author={Buchmann, Johannes},
       title={Selecting parameters for secure mceliece-based cryptosystems},
        date={2012Jun},
        ISSN={1615-5270},
     journal={International Journal of Information Security},
      volume={11},
      number={3},
       pages={137\ndash 147},
         url={https://doi.org/10.1007/s10207-011-0153-2},
}

\bib{pellikaan1993efficient}{incollection}{
      author={Pellikaan, Ruud},
       title={On the efficient decoding of algebraic-geometric codes},
        date={1993},
   booktitle={Eurocode’92},
   publisher={Springer},
       pages={231\ndash 253},
}

\bib{peterson1960encoding}{article}{
      author={Peterson, Wesley},
       title={Encoding and error-correction procedures for the bose-chaudhuri
  codes},
        date={1960},
     journal={IRE Transactions on Information Theory},
      volume={6},
      number={4},
       pages={459\ndash 470},
}

\bib{pinero2014subfield}{article}{
      author={Pi{\~n}ero, Fernando},
      author={Janwa, Heeralal},
       title={On the subfield subcodes of {H}ermitian codes},
        date={2014},
     journal={Designs, codes and cryptography},
      volume={70},
      number={1-2},
       pages={157\ndash 173},
}

\bib{sakata1995fast}{article}{
      author={Sakata, Shajiro},
      author={Justesen, J{\o}rn},
      author={Madelung, Y.},
      author={Jensen, H.~Elbrond},
      author={Hoholdt, Tom},
       title={Fast decoding of algebraic-geometric codes up to the designed
  minimum distance},
        date={1995},
     journal={IEEE Transactions on Information Theory},
      volume={41},
      number={6},
       pages={1672\ndash 1677},
}

\bib{sakata1990extension}{article}{
      author={Sakata, Shojiro},
       title={Extension of the berlekamp-massey algorithm to n dimensions},
        date={1990},
     journal={Information and Computation},
      volume={84},
      number={2},
       pages={207\ndash 239},
}

\bib{ShProb1}{book}{
      author={Shiryaev, Albert~N.},
       title={Probability. 1},
     edition={3},
      series={Graduate Texts in Mathematics},
   publisher={Springer, New York},
        date={2016},
      volume={95},
        ISBN={978-0-387-72205-4},
        note={Translated from the fourth (2007) Russian edition by R. P. Boas
  and D. M. Chibisov},
      review={\MR{3467826}},
}

\bib{shor1999}{article}{
      author={Shor, Peter~W.},
       title={Polynomial-time algorithms for prime factorization and discrete
  logarithms on a quantum computer},
        date={1999},
     journal={SIAM Review},
      volume={41},
      number={2},
       pages={303\ndash 332},
      eprint={https://doi.org/10.1137/S0036144598347011},
         url={https://doi.org/10.1137/S0036144598347011},
}

\bib{skorobogatov1990decoding}{article}{
      author={Skorobogatov, Alexei~N.},
      author={Vladut, Serge~G.},
       title={On the decoding of algebraic-geometric codes},
        date={1990},
     journal={IEEE Transactions on Information Theory},
      volume={36},
      number={5},
       pages={1051\ndash 1060},
}

\bib{stepanov2012codes}{book}{
      author={Stepanov, Serguei~A.},
       title={Codes on algebraic curves},
   publisher={Springer Science \& Business Media},
        date={2012},
}

\bib{stichtenoth2009algebraic}{book}{
      author={Stichtenoth, Henning},
       title={Algebraic function fields and codes},
   publisher={Springer Science \& Business Media},
        date={2009},
      volume={254},
}

\bib{randmat2010erasure}{article}{
      author={Studholme, Chris},
      author={Blake, Ian~F.},
       title={Random matrices and codes for the erasure channel},
        date={2010},
        ISSN={0178-4617},
     journal={Algorithmica},
      volume={56},
      number={4},
       pages={605\ndash 620},
         url={https://doi.org/10.1007/s00453-008-9192-0},
      review={\MR{2581065}},
}

\bib{matlabstat}{manual}{
      author={The~MathWorks, Inc.},
       title={Statistics and machine learning toolbox},
     address={Natick, Massachusetts, United State},
        date={2019},
         url={https://www.mathworks.com/help/stats/},
}

\bib{V05}{article}{
      author={V{\'e}ron, P.},
       title={Proof of conjectures on the true dimension of some binary {G}oppa
  codes},
        date={2005},
        ISSN={0925-1022},
     journal={Des. Codes Cryptogr.},
      volume={36},
      number={3},
       pages={317\ndash 325},
         url={https://doi.org/10.1007/s10623-004-1722-4},
      review={\MR{2162582}},
}

\end{biblist}
\end{bibdiv}

\end{document}